\documentclass{article}
\usepackage{graphicx}
\usepackage{amssymb}
\usepackage{amsthm}
\usepackage{color}
\usepackage{amsmath}
\usepackage{fullpage}
\usepackage{setspace}
\usepackage{charter}

\usepackage{hyperref} 
\hypersetup{backref, colorlinks=true, citecolor=blue, linkcolor=blue}
\newcommand{\rref}[1]{\hyperref[#1]{\ref*{#1}}}

\newtheorem{theorem}{Theorem}[section]
\newtheorem{lemma}[theorem]{Lemma}

\newtheorem{conjecture}[theorem]{Conjecture}

\theoremstyle{definition}

\newcommand{\C}{\mathcal C}
\newcommand{\Z}{\mathcal Z}
\newcommand{\E}{\mathcal E}
\newcommand{\A}{\mathcal{A}}

\usepackage[round, authoryear]{natbib}

\begin{document}

\title{On the Existence of the MLE for a Directed Random Graph Network Model with Reciprocation
\author{Alessandro Rinaldo\thanks{Email: {\tt arinaldo@stat.cmu.edu}}\\ 
Department of Statistics\\
Carnegie Mellon University\\
Pittsburgh, PA, 15213-3890 USA  
\and
Sonja Petrovi\'c\thanks{Email: {\tt petrovic@math.uic.edu}}\\
Department of Mathematics, Statistics, and Computer Science\\
University of Illinois\\
Chicago, IL, 60607
\and
Stephen E. Fienberg\thanks{Email: {\tt fienberg@stat.cmu.edu}}\\
Department of Statistics\\
 Machine Learning Department\\
  Cylab\\
Carnegie Mellon University\\
Pittsburgh, PA, 15213-3890 USA}
\date{}
}

\maketitle

\begin{abstract}
\cite{holl:lein:1981}  proposed  a directed random graph model, the $p_1$ model,  to describe  dyadic interactions in a social network.   In previous work \citep{petr:rina:fien:2009}, we studied the algebraic properties of the $p_1$ model and showed that it is a toric model specified by a multi-homogeneous ideal.  We conducted an extensive study of the Markov bases for $p_1$ that incorporate explicitly the constraint arising from multi-homogeneity. 
Here we consider the properties of the corresponding  toric variety and relate them to the conditions for the existence of the maximum likelihood and extended maximum likelihood estimators or the model parameters. Our results are directly relevant to the estimation and conditional goodness-of-fit testing problems in $p_1$ models.\\

\noindent {\bf Keywords}:  Algebraic statistics, Dyadic independence, Exponential random graph model,  Holland-Leinhardt $p_1$ model, Markov basis.
\end{abstract}


\section{Introduction}
\label{section:introduction}

\cite{holl:lein:1981}  derived a statistical model, called $p_1$, for directed random graphs to describe  dyadic interactions in a social network.  The $p_1$ model, one of the earlier and most relevant statistical model for the analysis of  network data,  remains popular in applications, but some of its statistical properties are not well understood, including a relevant asymptotics.   
In the $p_1$ model, we represent the network by a directed graph where for every pair of nodes $(i,j)$  we can observe four possible configurations: an edge from $i$ to $j$, an edge from $j$ to $i$, a bi-directed edge (or both directed edges), or no edge at all.   Each pair of nodes, or {\it dyad}, is in any of these states independently of the other dyads. 
Although the $p_1$ model is  a special case of the broader and well understood class of log-linear models, e.g., see~\cite{fien:wass:1981a,fien:meye:wass:1985}, it possesses unique features that pose challenges of both theoretical and practical nature. We described some of these in \cite{petr:rina:fien:2009}, where
 we revisited the Holland-Leinhardt $p_1$ model using the tools and language  of  algebraic statistics \citep[see, e.g.,][]{diac:stur:1998,pist:ricco:wynn:2001,drto:stur:sull:2009}.
The results from that paper are summarized in Section~\ref{sec:background}. 
Here we focus on the underlying geometry of $p_1$ models, rather then their algebraic properties. Our main result is  Theorem~\ref{thm:mle}, which gives explicit conditions for the existence of the maximum likelihood estimator and offers algorithmic improvement in the study of the existence of MLE. 
In Section~\ref{sec:computations} we carry out some computations for small networks to exemplify our results and illustrate further the subtleties of the estimation problem.

\section{The  $p_1$ model: notation, description, and structure}\label{sec:p1.model}
\label{sec:background}

Versions of the  $p_1$ model prescribe probability distributions over the set of directed graphs on $n$ nodes with random directed and bi-directed edges. Each node corresponds to a unit in a given population of interest, and the edges  are random variables that represent relations or interactions between units. The resulting random graph provides a static snapshot of the interactions among a set of agents in the populations, and it is often referred to as a network.  

In $p_1$ models, the random edge between any pair of nodes $i$ and $j$, or {\it dyad}, is  modeled independently from all the others.  For any dyad defined by the pair of nodes $(i,j)$ there are four possible edge configurations: 
\begin{enumerate}
\item node $i$ has an outgoing edge into node $j$: $i \rightarrow j$;
\item node $i$ as an incoming edge originating from node $j$: $i \leftarrow j$;
\item there is a bi-directed edge between node $i$ and node $j$: $i \longleftrightarrow j$;
\item there is no edge between nodes $i$ and $j$.
\end{enumerate}
Following the notation we established in  \cite{petr:rina:fien:2009}, which is slightly different than the original notation of  \cite{holl:lein:1981}, for every pair of nodes $(i,j)$ we define the vector
\begin{equation}\label{eq:pij}
p_{i,j} = \left( p_{i,j}(0,0),p_{i,j}(1,0),p_{i,j}(0,1),p_{i,j}(1,1) \right) \in \Delta_3
\end{equation}
containing the probabilities of the four possible edge types, where $\Delta_3$ is the standard simplex in $\mathbb{R}^4$. The numbers $p_{i,j}(1,0)$, $p_{i,j}(0,1)$ and $p_{i,j}(1,1)$ denote the probabilities of the edge configurations $i \rightarrow j$, $i \leftarrow j$ and  $i \longleftrightarrow j$, respectively, and $p_{i,j}(0,0)$ is the probability that there is no edge between $i$ and $j$ (thus, $1$ denotes the outgoing side of the edge).  Notice that, by symmetry $p_{i,j}(a,b) = p_{j,i}(b,a)$, for any $a,b \in \{ 0,1\}$ and that 
\begin{equation}\label{eq:sum.1}
p_{i,j}(0,0) + p_{i,j}(1,0) + p_{i,j}(0,1) + p_{i,j}(1,1) = 1.
\end{equation}
The fundamental assumption underlying $p_1$ models is that the dyads are independent. This is formalized by modeling each of the ${n \choose 2}$ dyads as mutually independent draws from multinomial distributions with class probabilities $ p_{i,j}$ for $i < j$.
Specifically, the Holland-Leinhardt  $p_1$ model specifies the multionomial probabilities of each dyad $(i,j)$ in logarithmic form as follows:
\begin{equation}\label{eq:hl.eqs}
\begin{array}{rcl} 
    \log p_{i,j}(0,0) &= &\lambda _{i,j}\\
    \log p_{i,j}(1,0) &= & \lambda _{i,j} + \alpha_i + \beta_j + \theta \\ 
    \log p_{i,j}(0,1) &= & \lambda _{i,j} + \alpha_j + \beta_i + \theta \\
    \log p_{i,j}(1,1) &= & \lambda _{i,j} + \alpha_i + \beta_j + \alpha_j + \beta_i + 2\theta + \rho_{i,j}.
\end{array}
\end{equation}
The real numbers $\theta$, $\alpha_i$, $\beta_i$, $\rho_{i,j}$ and $\lambda_{i,j}$ for all $i< j$ are the model parameters. The parameter $\alpha_i$ quantifies the effect of an outgoing edge from node $i$, the parameter $\beta_j$ instead measures the effect of an incoming edge into node $j$, while $\rho_{i,j}$ controls the added effect of reciprocated edges (in both directions). The  parameter $\theta$ measures the propensity of the network to have edges and, therefore, controls the ``density" of the graph. The parameter $\lambda_{i,j}$ functions as a normalizing constant to ensure that \eqref{eq:sum.1} holds for each each dyad $(i,j)$. Note that, in order for the model to be identifiable, additional linear constraints need to be imposed on its parameters. We refer the interested readers to the original paper on $p_1$ model by \cite{holl:lein:1981} for an extensive interpretation of the model parameters.

As noted in \cite{fien:wass:1981}, different variants of the $p_1$ model can be obtained by constraining the model parameters. In \cite{petr:rina:fien:2009} we consider three special cases of the basic $p_1$ model, which differ in the way the reciprocity parameter is  modeled:
\begin{enumerate}
  \item  $\rho_{i,j}=0$, {\it no reciprocal effect};
    \item $\rho_{i,j}=\rho$, {\it constant reciprocation};
    \item $\rho_{i,j}=\rho + \rho_i+\rho_j$, {\it edge-dependent reciprocation}.
\end{enumerate}

As it is often the case with network data, we assume that data become available in the form of {\it one} observed network. Thus, each dyad $(i,j)$ is observed in only one of its four possible states and this one observation is a random vector in $\mathbb{R}^4$ with a $\mathrm{Multinomial}(1,p_{i,j})$ distribution. As a result, data are sparse and, even though the dyadic probabilities are strictly positive according to the defining equations \eqref{eq:hl.eqs}, only some of the model parameters may be estimated from the data. 

For a network on $n$ nodes, we represent the vector of $2n(n-1)$ dyadic probabilities as
\[
p=(p_{1,2}, p_{1,3}, \ldots,p_{n-1,n}) \in \mathbb{R}^{2n(n-1)},
\]
where, for each $i <j$, $p_{ij}$ is given as in \eqref{eq:pij}. The $p_1$ \emph{model} is the set of all probability distributions that satisfy the Holland-Leinhardt equations \eqref{eq:hl.eqs}.
Give the settings just described, the $p_1$ model is log-linear; that is, the set of candidate probabilities $p$ is such that $\log p$ is in the linear space spanned by the rows of some {design matrix} $A$, which can be constructed as follows (this construction is by no means unique). The $2 n (n-1)$ columns of
$A$ are indexed by the entries of the vectors $p_{i,j}$, $i < j$, where the $p_{i,j}$'s are ordered lexicographically,  and its rows by the model parameters, ordered arbitrarily. The $(r,c)$ entry of $A$ is equal to the coefficient of the $c$-th parameter in the logarithmic expansion of the $r$-the probability as indicated in \eqref{eq:hl.eqs}. In particular, notice that the entries of $A$ can only be $0$, $1$ or $2$.
For  example, in the case $\rho_{ij}= \rho + \rho_i + \rho_j$, the matrix $A$ has ${n \choose 2} + 3n + 3$ rows. When $n=3$, the design matrix corresponding to this model is

\[
\begin{array}{l}
 \lambda_{12}\\ 
 \lambda_{13}\\ 
 \lambda_{23}\\ 
 \theta\\ 
 \alpha_1 \\ 
 \alpha_2 \\ 
 \alpha_3\\ 
 \beta_1 \\ 
 \beta_2 \\ 
 \beta_3 \\ 
 \rho \\ 
 \rho_1 \\ 
 \rho_2 \\ 
 \rho_3\\
\end{array}
\begin{array}{|cccccccccccc}
\multicolumn{4}{c|}{p_{1,2}} & \multicolumn{4}{c|}{p_{1,3}} & \multicolumn{4}{c}{p_{2,3}}\\
\hline
     1 &    1 &    1 &    1 &    0 &    0 &    0 &    0 &    0 &     0 &     0 &     0 \\
    0 &    0 &    0 &    0 &    1 &    1 &    1 &    1 &    0 &     0 &     0 &     0 \\
    0 &    0 &    0 &    0 &    0 &    0 &    0 &    0 &    1 &     1 &     1 &     1 \\
    0 &    1 &    1 &    2 &    0 &    1 &    1 &    2 &    0 &     1 &     1 &     2 \\
    0 &    1 &    0 &    1 &    0 &    1 &    0 &    1 &    0 &     0 &     0 &     0 \\
    0 &    0 &    1 &    1 &    0 &    0 &    0 &    0 &    0 &     1 &     0 &     1 \\
    0 &    0 &    0 &    0 &    0 &    0 &    1 &    1 &    0 &     0 &     1 &     1 \\
    0 &    0 &    1 &    1 &    0 &    0 &    1 &    1 &    0 &     0 &     0 &     0 \\
    0 &    1 &    0 &    1 &    0 &    0 &    0 &    0 &    0 &     0 &     1 &     1 \\
    0 &    0 &    0 &    0 &    0 &    1 &    0 &    1 &    0 &     1 &     0 &     1 \\
    0 &    0 &    0 &    1 &    0 &    0 &    0 &    1 &    0 &     0 &     0 &     1 \\
    0 &    0 &    0 &    1 &    0 &    0 &    0 &    1 &    0 &     0 &     0 &     0 \\
    0 &    0 &    0 &    1 &    0 &    0 &    0 &    0 &    0 &     0 &     0 &     1 \\
    0 &    0 &    0 &    0 &    0 &    0 &    0 &    1 &    0 &     0 &     0 &     1 \\
\end{array}
\]


Let $\zeta = \{ \lambda_{i,j}, \alpha_i, \beta_j, \theta, \rho, \rho_i, \rho_j, i < j\}$ be the vector of model parameters, whose dimension $d$ depends on the type of restrictions on the $p_1$. Then, using the design matrix $A$, one can readily verify that the Holland-Leinhardt equations have  the log-linear representation 
\[
\log p = A^\top \zeta.
\]
Equivalently, letting $a_k$ and $p_k$ be the $k$-th column of $A$ and the $k$-th entry of $p$, respectively,
\[
p_k = \prod_{l=1}^d (e^{\zeta(l)})^{a_k(l)},
\]
which, upon applying the exponential transformation point-wise, is equivalent to
\begin{equation}\label{eq:var}
p_{i,j}(a,b) =  e^{\lambda_{i,j}} (e^{\alpha_i})^{a} (e^{\alpha_j})^{b} (e^{\beta_i})^{b} (e^{\beta_j})^{a} (e^{\theta})^{a+b} (e^{\rho_{i,j}})^{\min(a,b)}, \quad \forall i < j , \;\; \forall a,b \in \{0,1\}.
\end{equation}

We denote by $M_A$ the $p_1$ model defined by the design matrix $A$. 
It is the relatively open set in the positive orthant of $\mathbb{R}^{2n(n-1)}$ of dimension  $\mathrm{rank}(A)$ consisting of ${n \choose 2}$ probability distrbutions satisfying \eqref{eq:var}.  

Let $\mathcal{X}_n \subset \mathbb{R}^{2n(n-1)}$ denote the sample space, i.e. the set of all observable networks on $n$ nodes. Then,  every point $x$ in the sample space $\mathcal X$ can be written as
\[
x = (x_{1,2},x_{1,3},\ldots,x_{n-1,n}), 
\]
where each of the  ${n \choose 2}$ subvectors $x_{i,j}$ is a vertex of $\Delta_3$. Notice that $|\mathcal{X}_n| = 4^{n(n-1)}$. \\

\noindent {\bf Remark.} This way of representing  a network on $n$ nodes as a highly-constrained $0/1$ vector of dimension $2n(n-1)$ may appear redundant. Indeed, as in \cite{holl:lein:1981}, we could more naturally represent an $n$-node network  using the $n \times n$ incidence matrix with $0/1$ off-diagonal entries, where the $(i,j)$ entry is $1$ is there is an edge from $i$ to $j$ and $0$ otherwise. While this representation is more intuitive and parsimonious (as it only requires $n(n-1)$ bits), whenever $\rho \neq 0$, the sufficient statistics for the reciprocity parameter are not linear functions of the observed network. As a  consequence, the ajacency matrix representation does not lead directly to a log-linear model. \\

For any $x \in \mathcal{X}_n$, each subvector $x_{i,j}$ is the realization of a random vector in $\mathbb{R}^4$ having Multinomial distribution with size $1$ and class probabilities $p^0_{i,j} \in \Delta_3$, where
\[
p^0 = (p^0_{1,2},p^0_{1,3},\ldots,p^0_{n-1,n})
\]
is an unknown vector in $M_A$ (thus, $p^0 > 0)$. Furthermore, \eqref{eq:hl.eqs} implies the Multinomial vectors representing the dyad configurations are mutually independent. Then,
for any point $x \in \mathcal{X}_n$, the {\it likelihood function} is the function $\ell \colon M_A \rightarrow [0,1]$ given by
\begin{equation}\label{eq:sup}
\ell_x(p) = \prod_{i<j} \left( \prod_{k=1}^4 p_{i,j}(k)^{x_{i,j}(k)}\right)
\end{equation}
and the {\it maximum likelihood estimate}, or MLE, of $p^0$ is 
\begin{equation}\label{eq:MLE}
\hat{p} = \mathrm{argmax}_{p \in M_A} \ell_x(p).
\end{equation}
Despite $\ell_x$ being smooth and concave on its domain for each $x \in \mathcal{X}_n$, there exist points $x \in \mathcal{X}_n$ for which the (unique) supremum of $\ell_x$ is achieved on the boundary of $M_A$, and, therefore, it will have some zero coordinates.  In this case, the MLE  of $p_0$ is said  not to exist. Indeed, if a vector $p$ with zero entries is to satisfy the Holland and Leinhardt likelihood equations, then some of the model parameters must be equal to $-\infty$. Furthermore, nonexistence of the MLE implies that only certain linear combinations of the natural parameters, or, equivalently, that only certain entries of $p$, are estimable \citep[see, e.g.,][]{ERG}. While it is well known that nonexistence of the MLE is a potential issue, very little progress has been made since the work of \cite{haberman:77} on the problems of deciding whether the MLE is non-existent and which parameters are estimable if the MLE does not exist.



\subsection{Markov bases}
\label{sec:markov} 

We  briefly recall the main results from \cite{petr:rina:fien:2009}, describing the Markov bases for the different versions of the $p_1$ model.
 
By a fundamental theorem of Markov bases \cite{diac:stur:1998}, these correspond to generating sets of toric ideals encoded by the design matrices. 
Our study of Markov moves was motivated by noticing special components of the design matrices. As the components that were relevant for that study make an appearance in our analysis of MLE, let us introduce them here. 

First, we distinguish the design matrices (``$A$" above) for two of the the three cases of the $p_1$ model: denote the $n$-node matrices by $\Z_n$ 
and $\E_n$ for zero 
and edge-dependent reciprocation, respectively. 
For each case, we consider the matrix of the {simplified} model obtained 
by  simply forgetting the normalizing constants $\lambda_{i,j}$.  Denote these simplified matrices  by 
 $\tilde\Z_n$ 
 and $\tilde\E_n$. Note that ignoring $\lambda_{i,j}$'s results in zero columns for each column indexed by $p_{i,j}(0,0)$,  and so we are effectively also ignoring all $p_{i,j}(0,0)$'s.
Hence, the matrices $\tilde\Z_n$ 
and $\tilde\E_n$ have ${{n}\choose{2}}$ less rows and ${{n}\choose{2}}$ less columns than $\Z_n$ 
 and $\E_n$, respectively.
The second special matrix will be denoted by $\A_n$ and is common to all three cases. It is obtained from 
$\tilde\Z_n$ 
 or $\tilde\E_n$ by ignoring the columns indexed by $p_{i,j}(1,1)$ for all $i$ and $j$, and then removing any zero rows.  

It turns out that the Markov basis of the toric ideal of the ``common submatrix" $\A_n$ essentially determines the Markov bases for the simplified models. More precisely, \cite[Theorem 3.2.]{petr:rina:fien:2009} states that the ideal of the simplified model defined by  ${\tilde\Z_n}$ can be decomposed into the ideal defined by $\A_n$ plus another ideal $T$.  Here, $T$ consists of moves that replace a bidirected edge between any pair of nodes $(i,j)$ by the two edges between them (from $i$ to $j$ and vice versa). 
Further, \cite[Theorem 3.4.]{petr:rina:fien:2009} describes the ideal of the simplified model defined by $\tilde\E_n$: it can be decomposed into a sum of the ideal defined by $\A_n$ and $Q$, where $Q$ consists of moves that replace two bidirected edges between two pairs of nodes $(i,j)$ and $(k,l)$ by two other bidirected edges between $(i,k)$ and $(j,l)$.



Our main result used a classical algebraic construction 
associating  to any graph $G$  a toric ideal $I_G$, as follows. Let $\phi$ be the map between polynomial rings defined by $\phi(x_{ij}) = v_iv_j$ for each edge $\{i,j\}$ of $G$. 
Its kernel $I_G$ is the {defining ideal of the edge subring} of $G$. 
We show that these ideals provide the basic building blocks for the essential Markov moves of the $p_1$ model. 

\begin{theorem}[\cite{petr:rina:fien:2009}, Theorem 3.6.]
\label{thm:essentialMarkovForModel}
	    The essential Markov moves for the $p_1$ model on $n$ nodes     can be obtained from the Graver basis of $I_{G}$     together with the Markov basis for the ideal $I_{K_n}$. 
   	  Here,    $K_n$ is the complete graph on $n$ vertices, and     $G$ is the  subgraph of the bipartite graph $K_{n,n}$ with edge set $\{v_i,w_j\}$ for $1\leq i\neq j\leq n$. 
\end{theorem}

\section{Maximum Likelihood Estimation in $p_1$ Models}\label{sec:mle}

In order to derive our results about existence of the MLE for $p_1$ models, we first need to describe the geometric properties of these models. These properties rely in a fundamental way on the re-parametrization of $p_1$ models as log-linear models. Though, as pointed out above, this parametrization is less parsimonious than the original one by \cite{holl:lein:1981}, it is more suitable to our geometric analysis.

\subsection{Geometric properties of $p_1$ models}

As we explained in Section \ref{sec:p1.model}, the statistical model specified by a $p_1$ model with design matrix $A$ is the set $M_A$ of all vectors satisfying the Holland and Leinhardt equations \eqref{eq:hl.eqs}. 
Let $V_A \subset \mathbb{R}^{2n(n-1)}$ be the set of non-negative points points satisfying \eqref{eq:var}. In fact, $V_A$ arises as the solution set of a system of polynomial equations, and, in the language of algebraic geometry, it is the the toric variety corresponding to the toric ideal $I_A$ \citep[see][]{CLO,St,petr:rina:fien:2009}. It is immediate to see that the closure of $M_A$ is precisely the set
\[
S_A := V_A \cap D_n,
\] 
where $D_n =  \left\{ \left(p_{1,2}, p_{1,3}\ldots,p_{n-1,n} \right) \colon p_{ij} \in \Delta_3, \forall i \neq j \right\}$, with $\Delta_3$ the standard $3$-simplex.  Any point in $S_A$ is comprised of ${n \choose 2}$ probability distributions over the four possible dyad configurations $(0,0)$, $(1,0)$, $(0,1)$ and $(1,1)$, one for each pair of nodes in the network. However, in addition to all strictly positive dyadic probability distributions satisfying \eqref{eq:hl.eqs}, $S_A$ contains also probability distributions with zero coordinates that are obtained as pointwise limits of points in $M_A$  \citep[see, for example,][]{gms:06}.

The {\it marginal polytope} of the $p_1$ model with design matrix $A$ is defined as
\[
P_A := \{ t \colon t = A p, p \in D_n\}.
\]
As we will see, $P_A$ plays a crucial role for deciding whether the MLE exists given an observed network $x$.
For our purposes, it is convenient to represent $P_A$ as a Minkowski sum of simpler polytopes (for background, see \cite{ZIE:95}). To see this, decompose the matrix $A$ as
\[
A = [A_{1,2} A_{1,3} \ldots A_{n-1,n}], 
\]
where each $A_{i,j}$ is the submatrix of $A$ corresponding to the four probabilities
\[
\{ p_{ij}(0,0),p_{ij}(1,0), p_{ij}(0,1),p_{ij}(1,1) \}
\] 
for the $(i,j)$-dyad. Then, denoting by $\sum$ the Minkowski sum of polytopes, it follows that 
\[
P_A = \sum_{i < j} \mathrm{conv}(A_{i,j}).
\]
We will concern ourselves with the boundary of $P_A$, which we will handle using the following theorem, due to \cite{gritzman.sturmfels:1993}. For a polytope $P \subset \mathbb{R}^d$ and a vector $c \in \mathbb{R}^d$, we write
\[
S(P ; c) := \{ x \colon x^\top c \geq y^\top c, \forall y \in P\}
\]
for the set of maximizers $x$ of the inner product $c^\top x$ over $P$.
\begin{theorem} \label{thm:GS:93}
\citep{gritzman.sturmfels:1993}. Let $P_1, P_2, \ldots, P_k$ be polytopes in $\mathbb{R}^d$ and let $P = P_1 + P_2 +  \ldots + P_k$. 
A nonempty subset $F$ of $P$ is a face of $P$ if and only if 
$F = F_1 + F_2 + \ldots + F_k$ for some face $F_i$ of $P_i$ such that there exists $c \in \mathbb{R}^d$  with $F_i = S(P_i ; c)$ for all $i$. Furthermore, the decomposition 
$F = F_1 + F_2 + \ldots + F_k$ of any nonempty face $F$ is unique. 
\end{theorem}

The marginal polytope $P_A$ and the set $S_A$ are closely related. Indeed, as shown in \cite{morton:2008},  $S_A$ and $P_A$ are homeomorphic.
\begin{theorem}\label{thm:moment.map}
\citep{morton:2008}. The map $\mu \colon S_A \rightarrow P_A$ given by
$
p \mapsto A p
$
is a homemorphism.
\end{theorem}

 This result is a generalization of the moment map theorem for toric varieties defined by homogeneous toric ideals  \citep[see][]{fulton:1993} and \cite{ewald:1996} to the  multi-homogeneous case.  In particular, it implies that $P_A = \{ t \colon t = A p, p \in S_A\}$.

\subsection{Existence of the MLE in $p_1$ models}

We focus on two important problems related to maximum likelihood estimation in $p_1$ models that have not been thoroughly investigated in both theory and practice. Our results should be compared, in particular, with the ones given in \cite{haberman:77}. The geometric properties of $S_A$ and $P_A$ established above are fundamental to our analysis.
 
We begin with a result justifying the name for the marginal polytope.
\begin{lemma}\label{lm:marg.pol}
The polytope $P_A$ is the convex hull of the set of all possible observable marginals; in symbols:
\[
P_A = \mathrm{conv}(\{ t \colon t = Ax, x \in \mathcal{X}_n \}).
\]
\end{lemma}

\begin{proof}
If $t = Ax$ for some  $x \in \mathcal{X}_n$, then, by definition, $t$ is in the Minkowski sum of $\mathrm{conv}(A_{i,j})$'s for all $i<j$, yielding that $\mathrm{conv}(\{ t \colon t = Ax, x \in \mathcal{X}_n \})$ is a subset of $P_A$. Conversely, by theorem \ref{thm:GS:93}, any extreme point $v$ of $P_A$ can be written as $v = v_{1,2} + v_{1,3} + \ldots + v_{n-1,n}$, where each $v_{i,j}$ is an extreme point of $\mathrm{conv}(A_{i,j})$ such that $v_{i,j} = S(\mathrm{conv}(A_{i,j});c)$ for some vector $c$ and all $i < j$. Since, for every $i \neq j$, the columns of $A_{i,j}$ are affinely independent, they are also the extreme points of $\mathrm{conv}(A_{i,j})$. Therefore, the extreme points of $P_A$ are a subset of $\{ t \colon t = Ax, x \in \mathcal{X}_n \}$, which implies that $P_A$ is a subset of $\mathrm{conv}(\{ t \colon t = Ax, x \in \mathcal{X}_n \})$. 
\end{proof}

Let $x$ denote the observed network. From standard statistical theory of exponential families \cite[see, e.g.,][]{BN,brown:86,lau:96}, the MLE $\hat{p}$ exists if and only if the vector of margins $t = Ax$ belongs to the relative interior of $P_A$, denoted by $\mathrm{ri}(P_A)$. Furthermore, it is always unique and satisfies the {\it moment equations,} which we formulate using the moment map of Theorem \ref{thm:moment.map} as 
\begin{equation}\label{eq:mle}
\hat{p} = \mu^{-1}(t).
\end{equation}

When the MLE does  not exists, the likelihood function still achieves its supremum at a unique point, which also satisfies the moment equations \eqref{eq:mle}, and is called the {\it extended MLE}. The extended MLE is a point on the boundary of $S_A$ and, therefore, contains zero coordinates. Although not easily interpretable in terms of the model parameters of the Holland and Leinhardt likelihood equations \eqref{eq:hl.eqs}, the extended MLE still provides a perfectly valid probabilistic representation of a network, the only difference being that certain network  configurations have a zero probability of occurring.

We are concerned with the following two problems that are essential for computing the MLE and extended MLE:
\begin{enumerate}
\item  decide whether the MLE exists, i.e. whether observed vector of margins $t$ belongs to $\mathrm{ri}(P_A)$, the relative interior of the marginal polytope; 
\item compute  $\mathrm{supp}(\hat{p})$, the support of $\hat{p}$, where $\mathrm{supp}(x) = \{ i \colon x_i \neq 0\}$. Clearly, this second task is nontrivial only when $t$ is a point on the relative boundary of $P_A$, denoted by $\mathrm{rb}(P_A)$, and we are interested in the extended MLE.
\end{enumerate}

Both problems require handling the geometric and combinatorial properties of the marginal polytope $P_A$ and of its faces, and to relate any point in $\mathrm{rb}(P_A)$ to the boundary of $S_A$. In general, this is challenging, as the combinatorial complexity of Minkowki sums could  be very high. Fortunately, we can simplify these tasks by resorting to a larger polyhedron with which it is easier to work.

Let $C_A = \mathrm{cone}(A)$ be the { marginal cone}  \cite[see][]{mle1:06}, which is easily seen to be the convex hull of all the possible observable sufficient statistics if there were no constraints on the number of possible observations per dyad. Notice that, since the columns of $A$ are affinely independent, they define the extreme rays of $C_A$. Following \cite{gms:06}, a set $\mathcal{F} \subseteq \{1,2,\ldots,2n(n-1) \}$ is called a \emph{facial set} of $C_A$ if there exists a $c$ such that 
\[
c^\top a_i = 0, \;\;\; \forall i \in \mathcal{F} \quad \text{and} \quad c^\top a_i < 0, \;\;\;  \forall i \not \in \mathcal{F},
\]
where $a_i$ indicates the $i$-th column of $A$.  Then, $F$ is face of $C_A$ if and only if $F = \mathrm{cone}(\{a_i \colon i \in \mathcal{F}\})$, for some facial set $\mathcal{F}$ of $C_A$, and that there is a one-to-one correspondence between the facial sets and the faces of $C_A$.

In our main result of this section, we show that the existence of the MLE can be determined using the simpler set $C_A$ and that the supports of the points on the boundary of $S_A$ are facial sets of $C_A$. 

\begin{theorem}\label{thm:mle}
Let $t \in P_A$. Then, $t \in \mathrm{ri}(P_A)$ if and only if $t \in \mathrm{ri}(C_A)$.
Furthermore, for every face $G$ of $P_A$ there exists one facial set $\mathcal{F}$ of $C_A$ such that $\mathcal{F} = \mathrm{supp}(p)$, where $p = \mu^{-1}(t)$, for each $t \in \mathrm{ri}(G)$.
\end{theorem}

\begin{proof}
By Lemma \ref{lm:marg.pol}, $P_A \subset C_A$. Thus, since both $P_A$ and $C_A$ are closed, $t \in \mathrm{ri}(P_A)$ implies $t \in \mathrm{ri}(C_A)$. 
For the converse statement, suppose $t$ belongs to a proper face $G$ of $P_A$. Then, by Theorem \ref{thm:GS:93},
\begin{equation}\label{eq:t}
t = t_{1,2} + t_{1,3} + \ldots + t_{n-1,n}
\end{equation}
where $t_{i,j}$ belongs to a face of $\mathrm{conv}(A_{i,j})$, for each $i < j$. We now point out two crucial features of $P_A$ that can be readily checked:
\begin{enumerate}
\item[(i)] the first ${n \choose 2}$ coordinates of any point in $P_A$ are all equal to $1$;
\item[(ii)] the first ${n \choose 2}$ coordinates of each $t_{i,j}$ are all zeros except one, which takes on the value  $1$.
\end{enumerate}
Suppose now that $t \in \mathrm{ri}(C_A)$. Because of (i) and (ii), there exists a point in $x \in D_n$ with strictly positive entries such that $t = A x$. In turn, this implies that
\[
t = t'_{1,2} + t'_{1,3} + \ldots + t'_{n-1,n},
\]
where $t'_{i,j} \in \mathrm{ri}(\mathrm{conv}(A_{i,j}))$ for each $i < j$, which contradicts \eqref{eq:t}. Thus, $t \in \mathrm{rb}(C_A)$.
To prove the second claim, notice that, because of the uniqueness of the decomposition in Theorem \ref{thm:GS:93}, our arguments further yield that, for every proper face $G$ of $P_A$, there exists one face $F$ of $C_A$ such that $\mathrm{ri}(G) \subseteq \mathrm{ri}(F)$. Consequently, to every face $G$ of $P_A$  corresponds one facial set $\mathcal{F}$ of $C_A$ such that, if $t \in \mathrm{ri}(G)$, then
\begin{equation}\label{eq:p}
t = Ap
\end{equation}
for some $p \in D_n$ with $\mathrm{supp}(p) = \mathcal{F}$. Since columns of $A$ are affinely independent, this $p$ is unique.  By Theorem \ref{thm:moment.map}, there exists a unique point $p' = \mu^{-1}(t) \in S_A$ satisfying \eqref{eq:p}, so that $\mathrm{supp}(p') \subseteq \mathcal{F}$. Inspection of the proof of the same theorem yields that, in fact, $\mathrm{supp}(p')= \mathcal{F}$, so that $p'=p$. See also the appendix in \cite{gms:06}.
\end{proof}

From the algorithmic standpoint, the previous theorem  allows us to work directly with $C_A$, whose face lattice is much simpler.  For example, while we know the extreme rays of $C_A$, it is highly nontrivial to find the vertices of $P_A$ among $4^{n(n-1)}$ vectors of observable sufficient statistics.
Algorithms for deciding whether $t \in \mathrm{ri}(C_A)$ and for finding the facial set corresponding to the face $F$ of $C_A$ such that $t \in \mathrm{ri}(F)$ are presented in \cite{mle1:06} and \cite{mle2}. 

We conclude this section with a statistical remark.  Using the terminology of log-linear modeling  theory  \citep[see, e.g.,][]{bfh:07}) $p_1$ models are log-linear models arising from a  {\it product-Multinomial sampling scheme.} This additional constraint is what makes dealing with $P_A$ \citep[and, as we showed in][computing the Markov bases]{petr:rina:fien:2009}  particularly complicated in comparison with the same tasks under Poisson or Multinomial sampling scheme. In these simpler sampling settings, the toric ideal $I_A$ is homogenous, the model $S_A$ is homoemorphic to $C_A$ for the Poisson scheme, and to $\mathrm{conv}(A)$ for the Multinomial scheme and all Markov moves are applicable (provided that their degree is smaller than $\sum_i x_i$ for the Multinomial scheme).

\subsection{Computations}
\label{sec:computations}
We now briefly summarize some of the computations involving $P_A$ and $C_A$ we carried out based on {\tt polymake} \citep[see][]{polymake}) and {\tt minksum} \citep[see][]{minksum}. 
Even though, due to the high computational complexity, this can only be a rather limited account, we believe it   provides some indications of the complexity of $p_1$ models and of the subtleties of the  maximum likelihood estimation problem.

We first indicate how non-existence of the MLE and the determination of the appropriate facial set can be addressed using simple linear programming. While checking for the existence of the MLE is relatively inexpensive, the second task is more demanding. We refer the interested readers to \cite{mle2} for further details and more efficient algorithms.
In order to decide whether the MLE exists it is sufficient to establish whether the observed sufficient statistics $t$ belong to the relative interior of $P_A$, which, by Theorem \ref{thm:mle}, happens if and only if it belongs to the relative interior of $C_A$. In turn, we can decide this by solving the following simple linear program:
\[
\begin{array}{rc}
& \max s  \\
\mathrm{s.t.} &  A x = t\\
& x_i - s \geq 0\\
& s \geq 0,\\
\end{array}
\]
where the scalar $s$ and vector $x$ are the variables. At the optimum $(s^*,x^*)$, the MLE exists if and only if $s^* > 0$. Though very simple, this algorithm may not be sufficient to compute the support of $\hat{p}$ if the MLE does not exist. To this end, we need to resort to a more sophisticated algorithm.  Let $a_i$ denote the $i$-th column of $A$, and consider the following $2 n (n-1)$ programs, one for each column of $A$:
\[
\begin{array}{rc}
& \max \langle a_i, y \rangle  \\
\mathrm{s.t.} &  y^\top t = 0\\
 &  A^\top y \geq 0\\
& -1 \leq y \leq 1
\end{array}
\]
Let $y^*_i$ denote the solution to the $i$-th program (notice that $y^*_i$ is a vector). Then,  the MLE does not exist if and only if  $\langle a_i, y^*i \rangle > 0$ for some $i$, in which case the facial set associated with $t$ is given by
\[
\{ i \colon \langle f_i, y^*_i\rangle = 0 \}.
\]

We now provide some numerical evidence on how Theorem \ref{thm:mle} significantly simplifies both tasks. Table \ref{tab:P_A} displays the number of vertices of the polytopes $P_A$ for the three model specifications described in Section~\ref{sec:p1.model} and various networks sizes. The last column of the table contains the number of columns of the design matrix, which is also the number of extreme rays of the marginal cone $C_A$. In comparison, the number of vertices of $P_A$ is very hard to compute; it is very large and grows extremely fast with $n$.

\begin{table}\label{tab:P_A}
\begin{center}
\begin{tabular}[t]{|c|c|c|c|c|}
\hline
$n$ & $\rho_{i,j}=0$ & $\rho_{i,j}=\rho$ & $\rho = \rho_i + \rho_j$ & $2 n (n-1)$\\
\hline
\hline
3 & 62 & 62 & 62 & 12\\
4 & 1,862 & 2,415 & 3,086 & 24\\
5 & 88,232 & 158,072 & 347,032 & 40\\
\hline
\end{tabular}
\end{center}
\caption{Number of vertices for the polytopes $P_A$ for different specifications of the $p_1$ model and different network sizes. Computations carried out using {\tt minksum} \cite{minksum}. The last column indicates the number of columns of the design matrix $A$, which correspond to the number of generators of $C_A$.}
\end{table}

In Table \ref{tab:C_A} we report the number of facets, dimensions and ambient dimensions of the cones $C_A$ for different values of $n$ and for the three specification of the reciprocity parameters $\rho_{i,j}$ we consider here. Though this only provides an indirect measure of the complexity of these models and of the non-zero patterns in extended MLEs, it does show how quickly the complexity of $p_1$ models may scale with the network size $n$. 

\begin{table}\label{tab:C_A}
\begin{center}
\begin{tabular}[t]{|c|ccc|ccc|ccc|}
\hline
$n$ & \multicolumn{3}{|c|}{$\rho_{i,j}=0$} & \multicolumn{3}{|c|}{$\rho_{i,j}=\rho$} & \multicolumn{3}{|c|}{$\rho = \rho_i + \rho_j$}\\
\hline
& Facets & Dim. & Ambient Dim. & Facets &  Dim. & Ambient Dim. & Facets &  Dim. & Ambient Dim.\\
\hline
3 & 30 & 7 & 9 & 56 & 8 & 10 & 15 & 10 & 13\\
4 & 132 & 12& 14& 348 & 13 & 15 & 148 & 16 & 19\\
5 & 660 & 18 & 20 & 3032 & 19 & 21 & 1775 & 23 & 26\\
6 & 3181 & 25 & 27 & 94337 & 26 & 28 & 57527 & 31 & 34\\
\hline
\end{tabular}
\end{center}
\caption{Number of facets, dimensions and ambient dimensions of the the cones $C_A$ for different specifications of the $p_1$ model and different network sizes.}
\end{table}


Another point of interest is the assessment of how often the MLE exists. In fact, because of the product Multinomial sampling constraint, nonexistence of the MLE is quite severe, especially for smaller networks. 
Below we report our finding:
\begin{enumerate}
\item  $n=3$.\\
The sample space consists of $4^3 = 64$ possible networks. When $\rho_{i,j} = 0$ for all $i$ and $j$, there are  $63$ different observable sufficient statistics, only one of which belongs to $\mathrm{ri}(P_A)$. Thus, only one of the $63$ observable sufficient statistics leads to the existence of the  MLE. The corresponding fiber contains only  two networks: \verb=0 0 1 0 0 1 0 0 0 0 1 0= and \verb=0 1 0 0 0 0 1 0 0 1 0 0=, which encode the incidence matrices
\[
\left[ 
\begin{array}{ccc}
\times & 0 & $1$\\
$1$ & \times & $0$\\
$0$ & $1$ & \times
\end{array}
\right] \quad \text{and} \quad 
\left[ 
\begin{array}{ccc}
\times & $1$ & 0\\
$0$ & \times & $1$\\
$1$ & $0$ & \times
\end{array}
\right],
\]
respectively. In both cases, the associated MLE is the $12$-dimensional vector with entries all equal to $0.25$.
Incidentally, the marginal polytope $P_A$ has $62$ vertices and $30$ facets.
When $\rho_{i,j} = \rho \neq 0$ or $\rho_{i,j} = \rho_i + \rho_j$ the MLE never exists. 
\item  $n=4$. \\
The sample space contains $4096$ observable networks. If $\rho_{i,j}=0$, there are $2,656$ different observable sufficient statistics, only $64$ of which yield existent MLEs. Overall, out of the $4096$ possible networks, only $426$ have MLEs.
When $\rho_{i,j}=\rho \neq 0$, there are $3,150$ different observable sufficient statistics, only $48$ of which yield existent MLEs. Overall, out of the $4096$ possible networks, only $96$ have MLEs. When $\rho_{i,j} = \rho_i + \rho_j$, there are $3150$ different observable sufficient statistics and the MLE never exists.
\item  $n=5$. \\
The sample space consists of  $4^{10} = 1,048,576$ different networks. If $\rho_{i,j}=0$, there are $225,025$ different sufficient statistics, and the MLE exists for $7,983$. If $\rho_{i,j} = \rho \neq 0$ the number of distinct possible sufficient statistics is $349,500$, and the MLE exists in $12,684$ cases. Finally, when $\rho_{i,j} = \rho_i + \rho_j$, the number of different sufficient statistics is $583,346$ and the MLE never exists.
\end{enumerate}

\subsubsection{The case $\rho = 0$}
When $\rho = 0$, nonexistence of the MLE can be more easily detected. Even though this is the simplest of $p_1$ models we consider, the observations below apply to more complex $p_1$ models as well, since nonexistence of the MLE for the case $\rho = 0$ implies nonexistence of the MLE in all the other cases. Furthermore, by setting $\rho = 0$, we recover the random  Rasch matrix model for exchangeable binary arrays discussed in \cite{lau:08}. Furthermore, the undirected version of this $p_1$ model corresponds to the random graph models with independent dyads and the degree sequence as the vector of sufficient statistics, recently studied by \cite{degree}.

Only for this case, it is more convenient to switch to the parametrization originally used in \cite{holl:lein:1981}, and describe networks using incidence matrices. Thus, each network on $n$ nodes can be represented as a $n \times n$ matrix with $0/1$ off-diagonal entries, where the $(i,j)$ entry is $1$ if there is an edge from $i$ to $j$, and $0$ otherwise. 
In this parametrization, the sufficient statistics for the parameters $\{ \alpha_i, i=1,\ldots,n\}$ and $\{ \beta_j, j=1,\ldots,n\}$ are the row and column sums, respectively. Just like with the other $p_1$ models, the minimal sufficient statistic for the density parameter $\theta$ is the total number of edges, which is a linear function of the sufficient statistics for the other parameters; because of this, it can be ignored.  

When $\rho = 0$, there are three cases where the MLE does not exist. The first two are immediate, while the third one is quite subtle. In order to describe them we recall that, from the standard theory of exponential families \cite[see, e.g.,]{lau:96}), the MLE $\hat{p}$  viewed as a $n \times n$ incidence matrix satisfies the moment equations, namely the row and column sums of $\hat{p}$ match the corresponding row and column sums of the observed network $x$. Thus, the MLE does not exist whenever this constraint cannot be satisfied by any strictly positive vector. As a result, for an observed network $x$, the MLE is non-existent if $x$ contains zeros in certain positions such that there does not exist any vector $z$ with strictly positive coordinates satisfying $A x = A z$.
To this end, we consider the $2n \times n(n-1)$ sub-matrix $B$ of the design matrix $A$ obtained by including only the columns corresponding to the probabilities $p_{i,j}(1,0)$ and $p_{i,j}(0,1)$, for all $i < j$ and only the rows corresponding to the parameters $\{\alpha_i, i=1,\ldots, n\}$ and $\{\beta_j, j=1\ldots, n\}$. For example, when $n=3$,
 \begin{align*}
   B = \begin{bmatrix} 
     1  &  0    &  1  &  0    &  0  &  0 \\
     0  &  1    &  0  &  0    &  1  &  0  \\
     0  &  0    &  0  &  1   &  0  &  1  \\
     0  &  1    &  0  &  1    &  0  &  0 \\
     1  &  0    &  0  &  0    &  0  &  1  \\
     0  &  0    &  1  &  0    &  1  &  0  
  \end{bmatrix},
\end{align*}
where its columns correspond to the probabilities $p_{1,2}(1,0)$, $p_{1,2}(0,1)$, $p_{1,3}(1,0)$, $p_{1,3}(0,1)$, $p_{2,3}(1,0)$ and $p_{2,3}(0,1)$. 
Since the columns of $B$ are affinely independent,  in order to determine the patterns of zeros leading to a nonexistent MLE, it is sufficient to look at the facial sets of the pointed polyhedral cone $\mathrm{cone}(B)$.

Using the strategy outlined above, we can distinguish three separate cases for which the MLE does not exist. The first two cases are obvious, since they correspond, respectively, to the maximal and minimal number of $1$'s that can be observed given the dyadic multinomial constraints. 
\begin{enumerate}
\item A row or column sum is equal to $n-1$. If row $i$ sums to $n-1$, then $\hat{p}_{ij}(1,0) = 1$ for all $j$, which implies that $\hat{p}_{ij}(0,1) = \hat{p}_{ij}(0,0) = 0$, for all $j$. Similarly, if column $i$ sums to $n-1$, $\hat{p}_{ij}(0,1) = 1$ for all $j$, which implies that $\hat{p}_{ij}(1,0)  = \hat{p}_{ij}(0,0) = 0$, for all $j$.
\item A row or column sum is $0$. If row $i$ has a zero sum, then  $\hat{p}_{ij}(1,0) = \hat{p}_{ij}(1,1) = 0$, for all $j$. Similarly, if column $i$ has a zero sum, then $\hat{p}_{ij}(0,1) = \hat{p}_{ij}(1,1) = 0$, for all $j$.
\item The last case is much less obvious and does not appear to be captured by a general rule, so we provide two instances of it for the cases $n=3$ and $n=4$.  When $n=3$, it is easy to see that a MLE with positive coordinate cannot exist if any of the following patterns of zeros are observed in any network $x$:
\[
\left[ 
\begin{array}{ccc}
\times & 0 & $\;$\\
0 & \times & $\;$\\
$\;$ & $\;$ & \times
\end{array}
\right], \quad 
\left[ 
\begin{array}{ccc}
\times & $\;$ & 0\\
$\;$ & \times & $\;$\\
0 & $\;$ & \times
\end{array}
\right], \quad
\left[ 
\begin{array}{ccc}
\times & $\;$ & $\;$\\
$\;$ & \times & 0\\
$\;$ &0  & \times
\end{array}
\right].
\]
Notice that the occurrence of a zero in the positions indicated above does not necessarily imply any of the previous two cases.
When $n=4$, there are 4 patterns of zeros leading to a nonexistent MLE, even though the margins can be positive and smaller than $3$:
\[
\left[ 
\begin{array}{cccc}
\times & 0 & $\;$ & 0\\
0 & \times & $\;$ & 0\\
$\;$ & $\;$ & \times & $\;$\\
0 & 0 & $\;$ & \times\\
\end{array}
\right], \quad 
\left[ 
\begin{array}{cccc}
\times & 0 & 0 & $\;$\\
0 & \times & 0 & $\;$\\
0 & 0 & \times & $\;$\\
$\;$ & $\;$ & $\;$ & \times\\
\end{array}
\right],\quad 
\left[ 
\begin{array}{cccc}
\times & $\;$ & 0 & 0\\
$\;$ & \times & $\;$ & $\;$\\
0 & $\;$ & \times & 0\\
0 & $\;$ & 0 & \times\\
\end{array}
\right], \quad
\left[ 
\begin{array}{cccc}
\times & $\;$ & $\;$ & $\;$\\
$\;$ & \times & 0 & 0\\
$\;$ & 0 & \times & 0\\
$\;$ & 0 & 0 & \times\\
\end{array}
\right].
\]
As it turns out, we can  show that that these are the same patterns of zeros that are relevant to the existence of the MLE for the model of quasi-independence for two-way contingency tables and the Rasch model in which the number of subjects equals the number of items.

\end{enumerate} 

Based on our computations  carried out in  {\tt polymake} for networks of size up to $n=10$, we have the following conjecture:

\begin{conjecture}For a network on $n$ nodes, $\mathrm{cone}(B)$ has  $3n$ facets, $2n$ of which correspond to patterns of zeros leading to a zero row or column margin, and the remaining $n$ to patterns of zeros which cause the MLE not to exist without inducing zero margins.\end{conjecture}


Finally, we point out that that the matrix $B$ described above plays a crucial role in describing the Markov bases for $p_1$ model as we described in our previous work \cite{petr:rina:fien:2009}, where we refer to $B$ as the \emph{common submatrix} ofall the $p_1$ models. Indeed, in Theorem
~\ref{thm:essentialMarkovForModel} and two other related results summarized in Section~\ref{sec:markov}, we showed that the Markov basis of the toric ideal generated by $B$ essentially determines the Markov bases for the simplified models.


\section*{Acknowledgements}
The authors are grateful to Xiaolin Yang for helping with the computations based on {\tt polymake}. The first and third authors were partially supported by NSF grant DMS-0631589.

\end{document}